\documentclass[12pt]{article}
\usepackage[left=3cm,right=3cm, top=2.5cm,bottom=2.5cm,bindingoffset=0cm]{geometry}
\usepackage{cmap}

\usepackage{amsmath,amsthm,amssymb}
\usepackage[T2A]{fontenc}
\usepackage[english]{babel}
\usepackage[cp1251]{inputenc}
\usepackage{amscd}

\usepackage{graphicx}

\usepackage{caption}
\usepackage{subcaption}

 \usepackage{cite}

\usepackage{hyperref}

\newtheorem{definition}{Definition}
\newtheorem{theorem}{Theorem}
\newtheorem{corollary}{Corollary}
\newtheorem{lemma}{Lemma}
\newtheorem{assertion}{Assertion}

\newtheorem{example}{Example}
\newtheorem{remark}{Remark}

\newcommand {\cF}   {\mathcal{F}}
\newcommand {\cV}   {\mathcal{V}}
\newcommand {\bbZ}  {\mathbb{Z}}
\newcommand {\bbR}  {\mathbb{R}}

\DeclareMathOperator   {\Aut}     {Aut}
\DeclareMathOperator   {\sgrad}   {sgrad}
\DeclareMathOperator   {\Sym}     {Sym}

\author{I.\,K.~Kozlov\thanks{No Affiliation,  E-mail: {\tt ikozlov90@gmail.com}} \quad  and \quad  A.\,A.~Oshemkov\thanks{Faculty of Mechanics and Mathematics, Moscow State University,
Moscow, 119991 Russia,  E-mail: {\tt a@oshemkov.ru} }}

\title{Classification of saddle-focus singularities}

\date{}

\begin{document}

\maketitle

\begin{abstract}
The paper presents an algorithm for topological classification of nondegenerate saddle-focus singularities of integrable Hamiltonian systems
with three degrees of freedom up to semi-local equivalence. In particular, we prove that any singularity of saddle-focus type 
can be represented as an almost direct product in which the acting group is cyclic. Based on the constructed algorithm, a complete list 
of singularities of saddle-focus type of complexity $1$, $2$, and $3$, i.e., singularities whose leaf contains one, two, or three singular 
points of rank $0$, is obtained. Earlier both singularities of saddle-focus type of complexity~$1$ were also described by L.\,M.~Lerman.
\end{abstract}

\begin{flushleft}

\textit{Bibliography:} 18 titles.

\end{flushleft}

\section{Introduction}

We study topological properties of singularities of integrable Hamiltonian systems in a neighborhood of a singular leaf.
More precisely, we consider saddle-focus singularities for systems with three degrees of freedom up to semi-local Liouville equivalence (i.~e., \ up to a leafwise homeomorphism of Liouville foliations in a neighborhood of a singular leaf). In particular, the paper presents an algorithm for constructing a complete list of saddle-focus singularities of given complexity $k$, i.e.\ with $k$ singular points of rank $0$ on the leaf (see Section~\ref{S:AlgAndComp1and2}).
Based on this algorithm, we obtain a classification of singularities of the saddle-focus type of complexity $\le3$
(see Theorems~\ref{Th:MainComp1},~\ref{Th:MainComp2},~\ref{Th:MainComp3}).

All necessary information about integrable systems and nondegenerate singularities can be found in \cite{BolsFom} or \cite{BolsOsh06}. We only briefly give the necessary definitions.

An {\em integrable Hamiltonian system} with $n$ degrees of freedom is a $2n$-dimensional symplectic manifold $(M^{2n},\omega)$ with $n$ functions $f_1,\dots,f_n$ defined on it such that the Hamiltonian vector fields $\sgrad f_i$ are complete, pairwise commute and are linearly independent almost everywhere on $M^{2n}$. The mapping $\Phi=(f_1,\dots,f_n):M\to\bbR^n$ is called the {\em momentum mapping} of an integrable Hamiltonian system $(M^{2n},\omega,f_1,\dots,f_n)$.

Note that the definition of a Hamiltonian system usually includes a {\em Hamiltonian}~--- some distinguished function $H$
from the set $f_1,\dots,f_n$ (or a function of $f_i$), whose Hamiltonian flow defines a dynamical system on the phase space $M^{2n}$. For us the choice of a Hamiltonian is not important, since we are not interested in the dynamics, but only in the topology of the foliation generated by the first integrals
$f_1,\dots,f_n$, which is defined as follows.

\begin{definition} 
The decomposition of the phase space of an integrable Hamiltonian system into connected components of $\Phi^{-1}(y)$ (i.e.\ the connected components of the common level surfaces of first integrals $f_1,\dots,f_n$) is called the {\em Liouville foliation} corresponding to this system.
\end{definition}

\begin{definition} 
Two integrable Hamiltonian systems on $U_1$ and $U_2$ are called {\em topologically equivalent} (or {\em Liouville equivalent}), if there exists a homeomorphism $\Psi\colon U_1\to U_2$ that maps each leaf of the Liouville foliation on $U_1$ to a leaf of the Liouville foliation on $U_2$.\end{definition}

Usually the Liouville foliation is not locally trivial and has singularities. In this paper, we consider Liouville foliations such that all their singularities are nondegenerate. These singularities are generic singularities for integrable systems. In fact, they are a multidimensional symplectic analogue of ``Morse singularities''. We present only a characteristic property
for nondegenerate singularities (Theorem~\ref{Th:Eliasson}), which can be taken as a definition
(See \cite{BolsFom} or \cite{BolsOsh06} for a precise definition).

There are three types of simplest ``basic'' nondegenerate singularities, they are given by the following Liouville foliations:
\begin{itemize}
\item $E$ is the Liouville foliation defined in a neighborhood of zero in $(\bbR^2,dp\wedge dq)$ by the function $p^2 +q^2$ (elliptic type);
\item $H$ is the Liouville foliation defined in a neighborhood of zero in $(\bbR^2,dp\wedge dq)$ by the function $pq$ (hyperbolic type);
\item $F$ is a Liouville foliation defined in a neighborhood of zero in $(\bbR^4,dp\wedge dq)$ by commuting functions $p_1q_1+p_2q_2$ and $p_1q_2-q_1p_2$
(focus-focus type).
\end{itemize}

If a Liouville foliation is given by functions $f_1,\dots,f_n$, then its singular points are those where the subspace generated by $\sgrad f_i$
has dimension less than~$n$. The {\em rank} of a singular point is the dimension of this subspace. In particular, the singular points of these foliations
Liouville $E$, $H$ and $F$ have rank~$0$.

\begin{theorem}[Eliasson; see\cite{Eliasson90, BolsFom}] \label{Th:Eliasson}
In a neighbourhood of a nondegenerate singular point of rank $r$, the Liouville foliation is locally leafwise symplectomorphic to the direct product of $k_e$ copies of the foliation $E$, $k_h$ copies of the foliation $H$, and $k_f$
copies of the foliation $F$ , as well as the trivial foliation $\bbR^r\times\bbR^r$.
\end{theorem}

The triple $(k_e,k_h,k_f)$ is called the {\em type} of a nondegenerate singular point. It is easy to see that the type of a singular point is uniquely defined.

A {\em singularity} of a Liouville foliation is the germ of the momentum map on a singular leaf, i.e.\ a leaf containing  singular points. A singularity is called {\em nondegenerate} (in particular, elliptic, hyperbolic, focus) if all its singular points are nondegenerate (and their type is elliptical, hyperbolic or focus, respectively).

The {\em rank of a singularity} is the minimum rank of points on it. For the singularities of almost direct product type considered in this paper (see Definition~\ref{Def:AlmostDirProd}) the types of all its minimum rank points are the same. This type $(k_e,k_h,k_f)$ is called the {\em type of a singularity}.

Let $W_1,\dots,W_l$ each be either a foliation without singularities or the simplest singularity: elliptic, hyperbolic, or focus-focus. Their product $W_1\times\dots\times W_l$ has a natural Liouville foliation with a singularity. This singularity is called a {\em direct product} singularity, and the singularities that are Liouville equivalent to it are singularities {\em of direct product type}.

\begin{definition} \label{Def:AlmostDirProd}  
An {\em almost direct product} of singularities is a quotient of a direct product singularity $W_1\times\dots\times W_l$
by an action $\psi$ of a finite group $G$ that satisfies the following conditions:
\begin{enumerate}
\item
the action $\psi$ on $W_1\times\dots\times W_l$ is component-wise, i.e.
$$
\psi(g)(x_1,\dots,x_l)=(\psi_1(g)(x_1),\dots,\psi_l(g)(x_l)),
$$
where $\psi_1,\dots,\psi_l$ are actions of the group $G$ on singularities $W_1,\dots,W_l$,
\item
the action on each factor $\psi_i(g):W_i\to W_i$ is a symplectomorphism that preserves the functions defining the Liouville foliation (and, in particular, the foliation itself),
\item
the action $\psi$ on $W_1\times\dots\times W_l$ is free.
\end{enumerate}
\par\noindent
A singularity that is Liouville equivalent to an almost direct product $(W_1\times\dots\times W_l)/G$
is called a singularity of {\em almost direct product type}.
\end{definition}

\begin{remark} \label{Rem:Zung}
Below we will consider only singularities of almost direct product type, since N.\,T.~Zung proved in \cite{Zung1} that
that any nondegenerate singularity that satisfies the non-splitability condition (which means that for each singular point
of minimum rank on a singular leaf its local bifurcation diagram coincides with the bifurcation diagram of the entire singularity),
is a singularity of almost direct product type. Most of the singularities that occur in integrable systems
in mechanics and physics are non-splittable (see~\cite{BolsFom}). However, singularities not
satisfying the non-splitability condition arise in systems that are invariant with respect to rotation (see, for example, \cite{KudrOsh20} as well as the references cited therein). Also one can construct ``artificial'' examples of ``splitable'' singularities (see~\cite{BolsOsh06}).

Note that Zung's theorem is topological in nature, i.e., almost direct products of singularities that are Liouville equivalent  need not be leafwise symplectomorphic. Since in this paper we consider singularities only up to Liouville equivalence, the singularities of almost direct product type will sometimes be called simply almost direct products.
\end{remark} 

There are many works in which the singularities of integrable Hamiltonian systems (or Liouville foliations) were studied from different points of view: a study of specific integrable systems, invariants of singularities (topological, smooth, symplectic), classification of singularities of some fixed type. Since the local classification of nondegenerate singularities is described by the Eliasson theorem (Theorem~\ref{Th:Eliasson}), here, when we speak about the classification of singularities, we usually mean their \textit{semi-local classification}, i.e.,\ up to leafwise homeomorphism
of Liouville foliations in a neighborhood of a singular leaf.

Note also that there is a well-developed theory describing the ``global'' properties of Liouville foliations (and in particular their singularities) on isoenergetic manifolds of integrable Hamiltonian systems with two degrees of freedom. This theory was developed in the works of A.\,T.~Fomenko (see \cite{Fom86, Fom88}). Its detailed exposition is contained in the book~\cite{BolsFom}.

Let us list the main results concerning the classification of nondegenerate singularities of rank $0$ from the semi-local point of view.
\begin{itemize}
\item
Singularities of type $(1,0,0)$ and $(0,1,0)$ of rank $0$ are well known. Their properties are explored in detail in \cite{BolsFom}
(see also the description of these singularities in terms of atoms in~\cite{BMF} and in terms of $f$-graphs in~\cite{OshMiran}).
In particular, there is a unique elliptic $2$-atom~$A$ corresponding to a minimax singularity, and all other $2$-atoms are saddle.
For example, in Fig.~\ref{Fig:AtomsAandB} and \ref{Fig:SaddleComp2Atoms} all $2$-atoms of complexity $1$ and $2$ are shown schematically
with their standard notation.

\item
A complete list of saddle-saddle type singularities (i.e. \ singularities of type $(0,2,0)$ and rank~$0$) has been described and studied in the works of L.\,M.~Lerman, Ya.\,L.~Umansky~\cite{LermanUmankii92}, A.\,V.~Bolsinov~\cite{Bols91}, V.\,S.~Matveev~ \cite{Matveev96}. It turns out that there are $4$ and $39$ topologically distinct saddle-saddle singularities for complexity $1$ and $2$, respectively.

\item
The complete answer in the purely saddle case, i.e., for singularities of type $(0,n,0)$ and rank $0$, was obtained by
A.\,A.~Oshemkov~\cite{Oshemkov10}. For example, there are $32$ singularities of type $(0,3,0)$ and rank $0$  for complexity $1$ (see\ \cite{Oshemkov11}).

\item
The classification of singularities of type $(0,0,m)$ and rank $0$ for an arbitrary complexity (i.e.\ in the purely focus case) was done by A.\,M.~Izosimov~\cite{Izosimov11}.

\item
Singularities of saddle-focus type of complexity $1$ were described by L.\,M.~Lerman (see\ \cite{Lerman2000}).

\end{itemize}

\begin{figure} 
\hbox to\hsize{%
\quad
\includegraphics[width=0.28\textwidth]{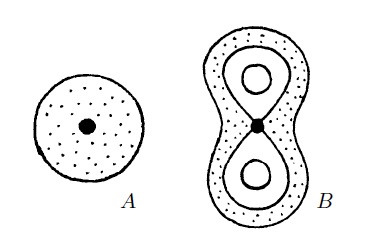}%
\hfill
\includegraphics[width=0.59\textwidth]{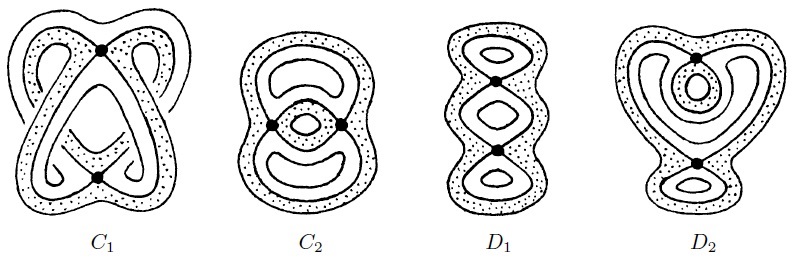}%
\quad}
\hbox to\hsize{%
\quad
\parbox{0.30\textwidth}{\caption{Atoms of complexity 1} \label{Fig:AtomsAandB}}%
\hfill
\parbox{0.59\textwidth}{\caption{Atoms of complexity $2$} \label{Fig:SaddleComp2Atoms}}%
\quad}
\end{figure}

In this paper, we will consider singularities of saddle-focus type (i.e., of type $(0,1,1)$) and rank~$0$. More precisely
(according to Remark~\ref{Rem:Zung}) we will consider singularities of almost direct product type
$(\cV_m\times\cF_n)/G$, where $\cV_m$ is some saddle $2$-atom of complexity~$m$, and $\cF_n$ is a focus singularity of complexity~$ n$. Following~\cite{Zung1}, an almost direct product will be called a {\em minimal model} of a saddle-focus singularity, if any (non-identity) element of $g\in G$ acts non-trivially on each component of the product $\cV_m\times\cF_n$. Any almost direct product can be reduced to its minimal model by taking the quotient of each factor by the action of the subgroup that trivially acts on another factor.

In N.\,T.~Zung's article~\cite{Zung1} it was stated that the minimal model for a singularity of almost direct product type is unique. Generally speaking, this is not true.

\addtocounter{example}{1}
\begin{example} 
Consider the $2$-atoms $B,D_1$ (see\ Fig.~\ref{Fig:AtomsAandB} and \ref{Fig:SaddleComp2Atoms}) and the regular Liouville foliation over the line
$W_{\mathrm{reg}}\approx S^1\times\bbR$. The $3$-atoms  $B\times S^1$ and $(D_1 \times S^1)/\bbZ_2$  corresponding to these $2$-atoms are Liouville equivalent (see~\cite{BolsFom}). The singularities of $B\times W_{\mathrm{reg}}$ and $(D_1\times W_{\mathrm{reg}})/\bbZ_2$ are also Liouville equivalent, although they are minimal models.
\end{example}

The statement about the uniqueness of the minimal model is true for some classes of singularities (for example, for elliptic or saddle
singularities of rank~$0$ (see~\cite{Oshemkov10}). However, as the following example shows, for saddle-focus singularities the minimum model is also not unique.

\begin{example} \label{Ex:SaddleFocusEquivModels} 
Minimal models $B\times\cF_1$ and $(D_1\times\cF_1)/\bbZ_2$, where the group $\bbZ_2$ acts on $\cF_1$ along the periodic integral as a shift of half a period are Liouville equivalent.
\end{example}

A more general statement (Example~\ref{Ex:SaddleFocusEquivModels} is a special case of it) is proved below in Lemma~\ref{L:HamActFact}.

We will prove that we can always ``get rid of'' such actions that lie in the continuous component of the automorphism group of focus singularities. Namely, for saddle-focus singularities we introduce the notion of a simple minimal model (see \ Definition~\ref{Def:SimpleMinModel})
and prove the following (see \ Theorems~\ref{Th:SimpleMinModelExists} and \ref{Th:SimpleMinModelUnique}):

\begin{itemize}
\item any singularity is Liouville equivalent to some simple minimal model;
\item
moreover, from any almost direct product $(\cV_m\times\cF_n)/G$ one can obtain a simple minimal model by taking the quotient by the action of some subgroup $N\triangleleft G$;
\item
for a simple minimal model the group $G$ is always cyclic;
\item
for a simple minimal model its factors $\cV_m$ and $\cF_n$, as well as the order of the group~$G$, are uniquely determined.
\end{itemize}

\section{A simple minimal model of a saddle-focus singularity}

\subsection{Existence of a simple minimal model}

\begin{definition} \label{Def:SimpleMinModel} 
We will call an almost direct product $(\cV_m\times\cF_n)/G$ {\em a simple minimal model} if any (non-trivial)
element $g\in G$
\begin{enumerate}
\item \label{Item:SimpleMinModelOldCond}
acts nontrivially on $\cV_m$ and $\cF_n$,
\item \label{Item:SimpleMinModelNewCond}
has no fixed points of rank $0$ on $\cF_n$.
\end{enumerate}
\end{definition}

This definition of a simple minimal model differs from the definition of a minimal model in N.\,T.~Zung's~\cite{Zung1}
by the addion of condition~(\ref{Item:SimpleMinModelNewCond}).

\begin{theorem} \label{Th:SimpleMinModelExists} 
Consider a singularity of saddle-focus type that is an almost direct product $(\cV_m\times\cF_n)/G$.
Let $N\subset G$ be the subgroup generated by elements that act trivially on the saddle component or
have fixed points on the focus component. Then
\begin{enumerate}
\item \label{Item:MainThGroup} the subgroup $N$ is normal and $G/N\cong\bbZ_k$,
\item \label{Item:MainThQuot}
$(\cV_m\times\cF_n)/N$ is leafwise symplectomorphic to the direct product $\cV'_{m'}\times\cF_{n'}$, where $V'_{m'}=V_m/N$,
and $n'$ is the number of orbits of the action of the subgroup~$N$ on the set of points of rank~$0$ of the focus singularity~$\cF_n$,
\item \label{Item:MainThSimMinModel}
$(\cV'_{m'}\times\cF_{n'})/\bbZ_k$ if a simple minimal model for the singularity $(\cV_m\times\cF_n)/G$.
\end{enumerate}
\end{theorem}

\begin{proof}[\indent Proof of Theorem~\ref{Th:SimpleMinModelExists}] 
It is obvious that $N$ is a normal subgroup of $G$. Let us show that the almost direct product $(\cV_m\times\cF_n)/G$ can be ``factorized''  by~$N$,
i.e.\ we can replace $\cV_m\times\cF_n$ with $(\cV_m\times\cF_n)/N$ and $G$ with $G/N$.

It is clear that if in an almost direct product $(W_1\times W_2)/G$ the group~$G$ acts trivially on the component~$W_2$, then $(W_1\times W_2)/G$ is leafwise symplectomorphic to $(W_1/G)\times W_2$. Therefore, without loss of generality, we may assume that the
the almost direct product $(\cV_m\times\cF_n)/G$ is a minimal model (i.e.\ the condition~(\ref{Item:SimpleMinModelOldCond}) of Definitions~\ref{Def:SimpleMinModel} is satisfied
).

By assumption, $G$ is a subgroup of the automorphism group of a focus singularity $\Aut(\cF_n)$, i.e., the group of leafwise symplectomorphisms of singularity~$\cF_n$ into itself that are identical on the base of the foliation. This group is described in~\cite{Izosimov11}.
Let us recall some assertions proved there. Let $H_f$ be a subgroup of Hamiltonian automorphisms of the focus singularity
in the group $\Aut(\cF_n)$, and let $x_1,\dots,x_n$ be  singular points of the focus singularity numbered in order.

\addtocounter{assertion}{3}

\begin{assertion}[see {\cite[Assertion~6]{Izosimov11}}] \label{A:AutFocusGroup}
The group $\Aut(\cF_n)$ is isomorphic to $\bbZ_k\times H_f$, where $k$ is a divisor of $n$, $\bbZ_k=\langle a\rangle$,
$a(x_q)=x_{q+k\pmod n}$ for any singular point $x_q$.
\end{assertion}

\begin{corollary}[see {\cite[Assertion~7]{Izosimov11}}]  
All elements of finite order in the group $\Aut(\cF_n)$ lie in its ``compact part'' $\Aut_0(\cF_n)=\bbZ_k\times S^1$.
\end{corollary} 

\begin{corollary} \label{Cor:MinModelGroup} 
For any minimal model of a saddle-focus singularity, the group $G$ is either cyclic or isomorphic to $\bbZ_k\oplus\bbZ_l$.
More precisely, $G\cap H_f\cong\bbZ_l$ and $G/(G\cap H_f)\cong\bbZ_k$, where $k,l\ge1$ (here we formally set $\bbZ_1=\{0\}$).
\end{corollary}
To prove Theorem~\ref{Th:SimpleMinModelExists}, it remains to take the quotient of the singularity by the action of the group $G\cap H_f\cong\bbZ_l$, i.e.\ to prove the following statement.

\begin{lemma} \label{L:HamActFact} 
Any saddle-focus singularity of the form $(\cV_m\times \cF_n)/G$, where the group $G$ acts on $\cF_n$ by Hamiltonian automorphisms
(i.e., \ preserves all singular points of $\cF_n$) is leafwise symplectomorphic to the singularity $(\cV_m/G)\times \cF_n$.\end{lemma}

\begin{proof}[\indent Proof of Lemma~\ref{L:HamActFact}] 
Let us represent a $2$-atom $\cV_m$ in a form convenient for us. We want to cut
$\cV_m$ so that $\cV_m/G$ can be glued together from its parts. To do this, we use the following obvious assertion.

\begin{assertion} 
If a group $G$ acts freely on a $2$-atom $\cV_m$, then it can be cut into ``crosses'' $K_i$ and ``ribbons'' $L_j$ so that the group $G$
acted freely on sets of ``crosses'' and ``ribbons''.
\end{assertion} 

In the case under consideration, the group $G$ acts freely on $\cV_m$, since it freely acts on $\cV_m\times\cF_n$
and preserves all singular points of $\cF_n$. As an example, Fig.~\ref{Fig:ActionExample} shows a similar representation for the $2$-atom $D_1$,
on which  acts  the group~$\bbZ_2$(as a central symmetry), and the fundamental domain for this action is highlighted.

\begin{figure}[h]
 \centering
 \includegraphics[width=0.35\textwidth]{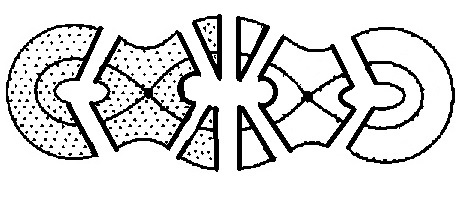}
 \caption{Fundamental region for the action of $\bbZ_2$ on $D_1$}
 \label{Fig:ActionExample}
\end{figure}

Let $\cV_m/G$ be obtained as a result of gluing along the boundary of some
``crosses'' $K_1,\dots,K_p$ and ``ribbons'' $L_1,\dots,L_q$.
Consider an arbitrary boundary point $x\in\partial\bigl(\bigl(\bigcup K_i\bigr)\cup\bigl(\bigcup L_j\bigr)\bigr)$.
Denote by $\sigma(x)$ the point with which it is identified. Then $(\cV_m/G)\times\cF_n$ is homeomorphic to the quotient space
$$
\Bigl(\bigcup K_i\times\cF_n\Bigr)\cup\Bigl(\bigcup L_j\times\cF_n\Bigr)\Big/\sim_1,
$$
where $(x,y)\sim_1(\sigma(x),y)$ for any point $y\in\cF_n$. In turn, $(\cV_m\times\cF_n)/G$ is homeomorphic to the quotient space
$$
\Bigl(\bigcup K_i\times\cF_n\Bigr)\cup\Bigl(\bigcup L_j\times\cF_n\Bigr)\Big/\sim_2,
$$
where $(x,y)\sim_2(\sigma(x),h_x(y))$ for any point $y\in\cF_n$, where $h_x$ is some Hamiltonian automorphism of the singularity~$\cF_n$.

Any Hamiltonian automorphism $h$ is homotopic to the identity in the class of Hamiltonian automorphisms (if $h$ is a shift by unit time along
Hamiltonian vector field~$v$, then the homotopy is a shift by time~$t$ along~$v$). Therefore $(\cV_m/G)\times\cF_n$ and $(\cV_m\times\cF_n)/G$ are homeomorphic,
since we can ``unwind'' each mapping~$h_x$ on a ``rectangle'' --- a small neighborhood of the boundary. For this you can use
the following assertion, which is easy to prove using the tubular neighborhood theorem.

\begin{assertion} 
If we cut the manifold $M^n$ along a submaifold $Q^{n-1}$ of codimension~$1$, and then glue the resulting boundary components
by a mapping homotopic to the original one, then we get a manifold that is homeomorphic to $M^n$.\end{assertion} 

This implies that $(\cV_m\times\cF_n)/G$ and $(\cV_m/G)\times\cF_n$ are Liouville equivalent. It remains to show that the Liouville equivalence
between them can be made symplectic. To do this, it suffices to note that the constructed homotopy automorphisms $h_x$ preserve the symplectic
structure. The fact that the homotopy of all Hamiltonian automorphisms $h_x$ can be made smoothly dependent on the point $x$ is guaranteed by the following simple
assertion.

\begin{assertion} 
Let $\mathcal{F}_n$ be a focus singularity, $f_1, f_2$ be its integrals such that the trajectories of the Hamiltonian vector field
$\sgrad f_2$ are closed with period $2\pi$. Then the functions $H_1 = H_1 (f_1, f_2)$ and $H_2 = H_2 (f_1, f_2)$ define the same Hamiltonian
automorphisms of $\mathcal{F}_n$, i.e. in other words, the $1$-shifts along the Hamiltonian vector fields $\sgrad H_1$ and $\sgrad H_2$ are equal,
if and only if
\[ H_2 - H_1 = \frac{c_1}{2\pi} f_2 + c_2,\] where $c_1 \in \mathbb{Z}, c_2 \in \mathbb{R}$.
\end{assertion}

Lemma~\ref{L:HamActFact} is proved.
\end{proof} 

After factorization of the minimal model by the action of the group $G\cap H_f$, we obtain a simple minimal model with the action of the group $G/(G\cap H_f)$.
Thus, item~\ref{Item:MainThSimMinModel} of Theorem~\ref{Th:SimpleMinModelExists} is fulfilled. Item~\ref{Item:MainThQuot} of
Theorem~\ref{Th:SimpleMinModelExists} is a consequence of Lemma~\ref{L:HamActFact}, Item~\ref{Item:MainThGroup} follows
from Corollary~\ref{Cor:MinModelGroup}. Theorem~\ref{Th:SimpleMinModelExists} is proved. 
\end{proof}

\subsection{Uniqueness of a simple minimal model}

\begin{theorem} \label{Th:SimpleMinModelUnique}
For any simple minimal model
$(\cV_m\times\cF_n)/\bbZ_k$
of a saddle-focus singularity
its $2$-atom $\cV_m$, focus singularity $\cF_n$ and group $\bbZ_k$ are uniquely determined
by the structure of the singularity itself.
\end{theorem}

We list some properties of critical points of a singularity of saddle-focus type in the form of the following statement,
which can be easily proved since we know how the group $\bbZ_k$ acts on a simple minimal model
(see Theorem~\ref{Th:SimpleMinModelExists}).

\begin{assertion} \label{A:SaddleFocusCritPoints}  
Consider a simple minimal model $(\cV_m\times\cF_n)/\bbZ_k$ of a saddle-focus singularity. Denote by $K_i$ the set of its critical
points of rank~$i$.
Then the following holds for a sufficiently small neighborhood of the singular leaf.
\begin{enumerate}
\item \label{Item:SaddleFocusPlaneAndLine}  
The bifurcation diagram, that is, the image of all critical points of the singularity under the momentum mapping is the union of a line~$l$ and a plane~$\Pi$,
transversely intersecting at a point~$P$.
\item
The set $K_0$ consists of $\frac{mn}k$ points lying in the inverse image of the point~$P$.
\item \label{Item:SaddleFocusCrit0and1}
The set $K_0\cup K_1$ forms a $2$-dimensional symplectic submanifold $M_1^2$ contained in the inverse image of the line~$l$.
The induced Liouville foliation on $M_1^2$ is Liouville equivalent to $\frac nk$ instances of the $2$-atom~$\cV_m$.

\item \label{Item:SaddleFocusCrit0and2} 
The set $K_0\cup K_2$ forms a $4$-dimensional symplectic submanifold $M_2^4$ contained in the inverse image of the plane~$\Pi$.
\begin{enumerate}
\item
If $k$ is odd, then the induced Liouville foliation on $M_2^4$ is Liouville equivalent to the union of   $\frac mk$  focus singularities~$\cF_n$.
\item
If $k=2k_1$ and the action of $\bbZ_k=\langle a\rangle$ on $\cV_m$ has exactly $s$ points of rank~$0$ with stabilizer $\bbZ_2$
(i.e., \ points fixed under the action of the element $a^{k_1}$), where $0\le s\le m$, then the induced Liouville foliation on $M_2^4$
is Liouville equivalent to the union of $\frac{m-s}k$ focus singularities~$\cF_n$ and $\frac s{k_1}$ focus singularities~$\cF_{\frac n2}$.
\end{enumerate}

\end{enumerate}
\end{assertion} 

Note the following fact, which follows from Item~\ref{Item:SaddleFocusCrit0and2} of Assertion~\ref{A:SaddleFocusCritPoints}.

\begin{corollary} \label{Cor:DiffFixPoints} 
If for two simple minimal models of the form $(\cV_m\times\cF_n)/\bbZ_k$
the action of $\bbZ_k$ on $\cV_m$ has a different number of points of rank~$0$ with stabilizer $\bbZ_2$, then these singularities are not Liouville equivalent.
\end{corollary} 

\begin{proof}[\indent Proof of Theorem~\ref{Th:SimpleMinModelUnique}] 
Let $(\cV_m\times\cF_n)/\bbZ_k$ be some simple minimal model for a given saddle-focus singularity.
Then according to Item~\ref{Item:SaddleFocusCrit0and1} of Assertion~\ref{A:SaddleFocusCritPoints} the structure of the induced Liouville foliation on the manifold $M_1^2=K_0\cup K_1$ uniquely determines the $2$-atom $\cV_m$ and also the ratio $\frac nk$.
		
To determine $n$ (that is, the factor $\cF_n$ in the minimal model), consider the induced Liouville foliation on $M_2^4=K_0\cup K_2$.
According to Item~\ref{Item:SaddleFocusCrit0and2} from Assertion~\ref{A:SaddleFocusCritPoints} this foliation is either a union of some $c$ focus singularities~$\cF_N$ for some $N$,
or a union of $a$ focus singularities~$\cF_N$ and $b$ focal singularities~$\cF_{2N}$, where $a,b\ge1$.
In the second case, we immediately get that $n=2N$ (and we can also easily calculate $s$, knowing $a$ and $b$).
But if the first case is realized, i.e.,\ the induced foliation on $M_2^4$ contains only focus singularities of the same complexity $N$,
then two options are theoretically possible: $s=0$ or $s=m$. In the first case $s=0$ we will have $n=N$ (and $k=\frac mc$), and in the second case  $s=m$ we get $n=2N$ (and~$k=\frac{2m}c$). Generally speaking, we cannot distinguish between these two options proceeding only from information about the induced Liouville foliations on the submanifolds $M_1^2$ and $M_2^4$. In other words, based on the Items~\ref{Item:SaddleFocusCrit0and1}~and~\ref{Item:SaddleFocusCrit0and2} from Assertion~\ref{A:SaddleFocusCritPoints},
we can distinguish any simple minimal models except for following pairs:

\begin{enumerate}
\item \label{Item:Zk} 
$(\cV_m\times\cF_n)/\bbZ_k$, where the action of any non-trivial element of the group $\bbZ_k$ has no fixed points of rank~$0$ on $\cV_m$,
\item \label{Item:Z2k}
$(\cV_m\times\cF_{2n})/\bbZ_{2k}$, where the involution $a^k$ of~$\cV_m$ (here $a$ is a generator of the group $\bbZ_{2k}$ )
leaves all $m$ points of rank $0$ fixed.
\end{enumerate}

It turns out that two actions with such conditions can exist only on $2$-atoms $\cV_m$ of some special form,
described in the following simple statement.

\begin{assertion} \label{A:UniqAtomIdentInvol} 
If there exists an involution on a $2$-atom $\cV_m$ that leaves all points of rank $0$ in place, then such  atom belongs to one of the two series of atoms shown in Fig.~\ref{Fig:XY-atoms} (their $f$-graphs are also shown there).
\end{assertion} 

\begin{figure}[h]
\centering
\includegraphics[width=0.95\textwidth]{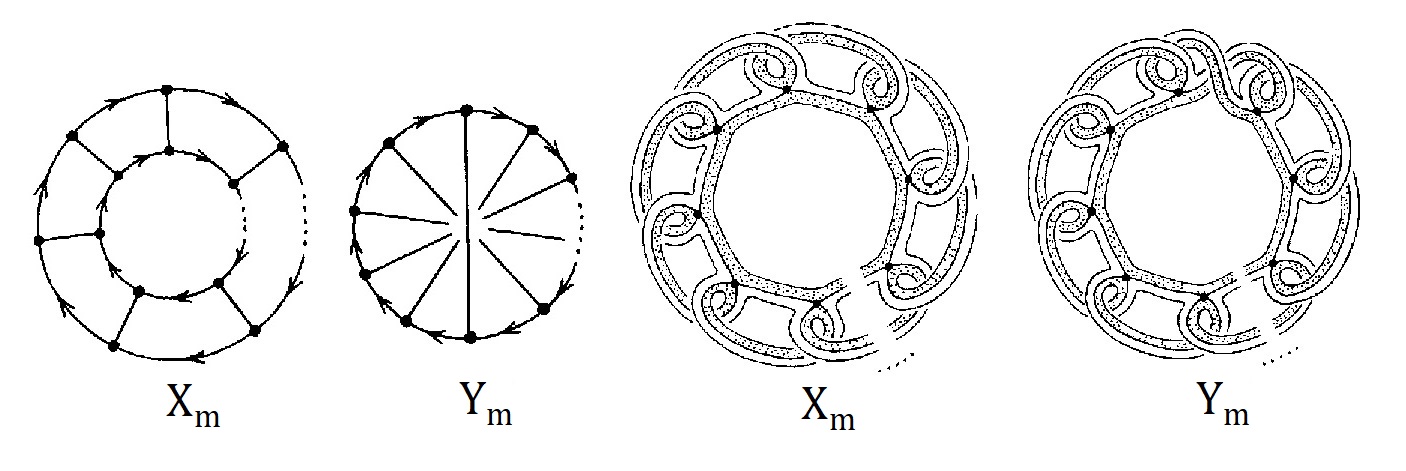}
\caption{Atoms $X_m$ and $Y_m$ with their $f$-graphs} 
\label{Fig:XY-atoms}
\end{figure}

In the book \cite{BolsFom} (where Fig~\ref{Fig:XY-atoms} comes from) these series of $2$-atoms
are given as examples of maximally symmetric atoms and are denoted by $X_m$ and $Y_m$, respectively.
Note that the properties of the $2$-atoms $X_m$ and $Y_m$ depend on the parity of the number~$m$ equal to their complexity.
In particular, the $2$-atoms $X_{2r}$ describe a bifurcation of two circles into two circles, the $2$-atoms~$Y_{2r}$~--- a bifurcation of one circle into one circle. The atoms $X_{2r+1}$ and $Y_{2r+1}$ differ by the sign of the function defining a foliation on them, and they describe a transformation of two circles into one or one into two, respectively.

As noted above, for simple minimal models of the form \ref{Item:Zk} and \ref{Item:Z2k}
the induced Liouville foliations on the manifolds $M_1^2$ and $M_2^4$ are the same. Let us consider other characteristics of these singularities, taking into account that from Assertion~\ref{A:UniqAtomIdentInvol} we know which 2-atoms $V_m$ can be in minimal models of the form \ref{Item:Zk} and \ref{Item:Z2k} .

For a singularity of the saddle-focus type, there is the following simple characteristic:
a pair of numbers equal to the number of Liouville tori in the inverse images of points in both semi-spaces into which the plane $\Pi$ in the image of the momentum mapping divides the space
(see Item~\ref{Item:SaddleFocusPlaneAndLine} Assertion~\ref{A:SaddleFocusCritPoints}). Considering all possible actions of the groups $\mathbb{Z}_k$ and $\mathbb{Z}_{2k}$ on $2$-atoms from the
series $X_m$ and $Y_m$ satisfying the conditions listed in \ref{Item:Zk} and \ref{Item:Z2k}, we get the following
statement.

\begin{assertion} Let $(\mathcal{V}_m \times \mathcal{F}_n)/Z_{k}$ and $(\mathcal{V}_m \times \mathcal{F}_{2n})/Z_{2k} $ be simple minimal
models of the form \ref{Item:Zk} and \ref{Item:Z2k}, where $\mathcal{V}_m$ is a 2-atom from the series $X_m$ or $Y_m$. Then the number of Liouville tori in the inverse images of regular points for the pair of half-spaces into which the plane $\Pi$ from the bifurcation diagram divides the image of the momentum map, is equal to

\begin{itemize}

\item(2, 2) for $(\mathcal{V}_m \times \mathcal{F}_n)/Z_{k}$ and (1, 1) for $(\mathcal{V}_m \times \mathcal{F }_{2n})/Z_{2k}$ if $\mathcal{V}_m = X_m$, where $m$ is even,
\item (1, 1) for $(\mathcal{V}_m \times \mathcal{F}_n)/Z_{k}$ and (1, 1) for $(\mathcal{V}_m \times \mathcal {F}_{2n})/Z_{2k}$ if $\mathcal{V}_m = Y_m$, where $m$ is even,
\item (2, 1) for $(\mathcal{V}_m \times \mathcal{F}_n)/Z_{k}$ and (1, 1) for $(\mathcal{V}_m \times \mathcal {F}_{2n})/Z_{2k}$ if $\mathcal{V}_m =X_m$ or $\mathcal{V}_m =Y_m$, where $m$ is odd.

\end{itemize}
\end{assertion}

Thus, it remains to prove that the singularities $(Y_m \times \mathcal{F}_n)/Z_{k}$ and $(Y_m \times \mathcal{F}_{2n})/Z_{2k}$, where $m$ is even, are not Liouville equivalent. To do this, it suffices to study the combinatorial structure of the singular leaf $L$ (i.~e., the inverse image of the point of intersection of the line $l$ and the plane $\Pi$ in the bifurcation diagram). Namely, we need to look at the stratification of the leaf $L$ by points
of different ranks, i.e., the sets $L \cap K_i$ and how they are adjoined to each other.

Every three-dimensional orbit in $L$ is a solid torus $\bbR^2 \times S^1$. Its boundary contains two one-dimensional orbits $\bbR^1$,
two two-dimensional orbits $\bbR^1 \times S^1$ and four zero-dimensional orbits (some of these orbits may coincide). We can
consider chains of adjoined three-dimensional and one-dimensional orbits. For considered singularities $(Y_m \times \cF_n)/\mathbb{Z}_k$ and $(Y_m \times \cF_{2n})/\bbZ_{2k}$
these chains will be different. For example, consider the case $n=1$ and $m=2$ (note that $Y_2$ is the
2-atom $C_1$, see Fig.~\ref{Fig:SaddleComp2Atoms}). $C_1 \times \cF_1$ has 4 chains, consisting of 1 one-dimensional orbit and 1 three-dimensional orbit.
And $\left(C_1 \times \cF_2\right) / \mathbb{Z}_2$ has 2 chains, in which there are 2 one-dimensional orbits and 2 three-dimensional orbits. In the general case, the proof is similar.

Theorem~\ref{Th:SimpleMinModelUnique} is proved. 
\end{proof} 

\section{Description of the algorithm and classification of singularities of low complexity} \label{S:AlgAndComp1and2}

\subsection{Symmetry group of a $2$-atom}

In order to describe saddle-focus singularities of low complexity, we briefly recall the structure of the symmetry group for $2$-atoms of low complexity (See \cite{BolsFom} for details).

Let $\mathcal{V}_m$ be an arbitrary $2$-atom. Denote by $\Sym(\mathcal{V}_m)$ the automorphism group of the $2$-atom $\mathcal{V}_m$ up to isotopy.
\cite{BolsFom} shows that this is a finite group, and that it is isomorphic to the automorphism group of the $f$-graph of the $2$-atom.
The Table~\ref{Tab:AtomSym} contains the group $\Sym(\mathcal{V}_m)$ for the atoms of complexity $m=1,2,3$.
A list of $2$-atoms of complexity 3 is given in \cite{BolsFom}. These atoms are shown schematically in Fig.~\ref{Fig:SaddleComp3}.

\begin{table}[h]
\centering
\begin{tabular}{|c|c|c|c|c|c|c|c|c|c|c|c|c|c|c|c|}
\hline
$m$           & \multicolumn{1}{c|}{1}  & \multicolumn{4}{c|}{2} & \multicolumn{10}{c|}{3} \\
\hline
atom $\cV_m$ &$B$     &$C_1$   &$C_2$               &$D_1$   &$D_2$&$E_1$   &$E_2$   &$E_3$&$F_1$&$F_2$&$G_1$   &$G_2$&$G_3$   &$H_1$   &$H_2$\\ 
\hline
$\Sym(\cV_m)$&$\bbZ_2$&$\bbZ_4$&$\bbZ_2\oplus\bbZ_2$&$\bbZ_2$&$e$  &$\bbZ_6$&$\bbZ_2$&$S_3$&$e$  &$e$  &$\bbZ_2$&$e$  &$\bbZ_2$&$\bbZ_3$&$e$  \\
\hline
\end{tabular}
\caption{Symmetry groups of saddle atoms of complexity 1, 2, 3}
\label{Tab:AtomSym}
\end{table}

\begin{figure}[h!]
 \centering
 \includegraphics[width=\textwidth]{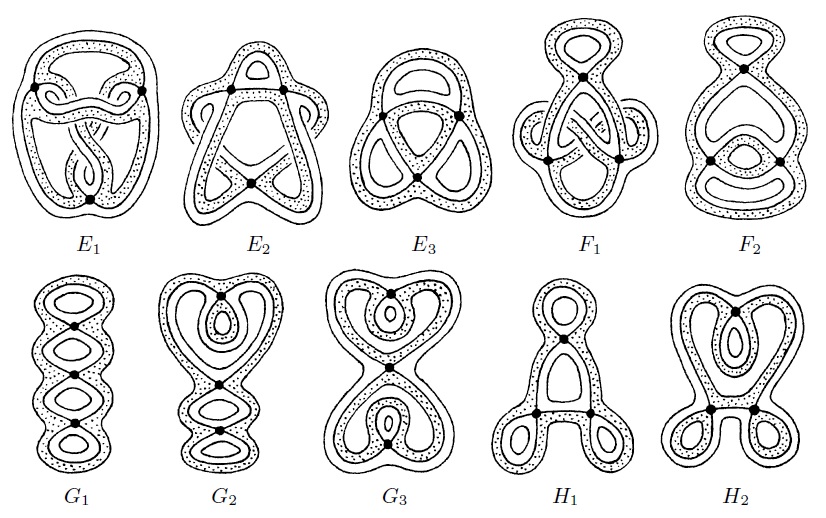}
 \caption{Atoms of complexity 3}
 \label{Fig:SaddleComp3}
\end{figure}

\begin{remark} \label{Rem:C2Sym} 
The $2$-atom $C_2$ has the symmetry group $\Sym(C_2)=\bbZ_2\oplus\bbZ_2$.
If we represent the $2$-atom $C_2$ symmetrically, as in Fig.~\ref{Fig:C2Sym}, then the nontrivial elements of $\Sym(C_2)$ are  rotations by $\pi$ around the axes $x$,$y$ and $z$. For the first two symmetries
there are no fixed points of rank $0$, the latter leaves the points of rank $0$ in place. \end{remark}

\begin{figure}[h!]
 \centering
 \includegraphics[width=0.35\textwidth]{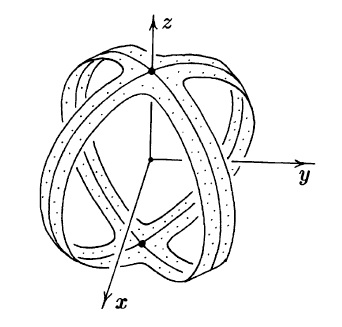}
 \caption{Symmetries of the atom $C_2$}
 \label{Fig:C2Sym}
\end{figure}

\subsection{Classification algorithm for saddle-focus singularities}

By Theorem~\ref{Th:SimpleMinModelExists}, to classify saddle-focus singularities, it suffices to consider their simple minimal models. They all have the form $(\cV_m\times\cF_n)/\bbZ_k$. Notice, that
\begin{itemize}
\item
the complexity of the singularity $(\cV_m\times\cF_n)/\bbZ_k$ is equal to $md$, where $n=kd$,
\item
the generating element of the group $\bbZ_k$ acts on $\cV_m$ as an element of order $k$ in the symmetry group $\Sym(\cV_m)$ of the atom $\cV_m$,
\item
the order of the symmetry group $\Sym(\cV_m)$ does not exceed $2m$.
\end{itemize}

Based on these considerations, we get the following algorithm to obtain a complete list of saddle-focus singularities of complexity $p$.
\begin{enumerate}
\item
We enumerate all possible divisors $m$ of the number $p$.
\item
For each saddle $2$-atom $\cV_m$ of complexity $m$, we consider all possible $k$ for which the group $\Sym(\cV_m)$ contains a subgroup $\bbZ_k$.
\item \label{Item:AlgConsModels}
We consider all possible simple minimal models of the form
\begin{equation} \label{Eq:AlgMinModel}
(\cV_m \times \cF_{\frac{kp}m}) / \bbZ_k,
\end{equation}
if they exist.
\item
If there are several simple minimal models of the same type \eqref{Eq:AlgMinModel}, then it is checked whether
they are Liouville equivalent.
\end{enumerate}

Let us show that for a singularity of any fixed complexity $p$, the algorithm will finish in a finite number of steps. It is obvious that there is a finite number of variants for numbers $m$, $k$ and $2$-atoms $\cV_m$ (up to Liouville equivalence). Therefore, it suffices to show that
at the step~\ref{Item:AlgConsModels} it is enough to consider a finite number of singularities. If the actions of the groups $\bbZ_k$ from \eqref{Eq:AlgMinModel} on any of the factors are conjugate
(by an automorphism of Liouville foliations), then the resulting fibrations are Liouville equivalent.

Note that for any simple minimal model $\left(\cV_m \times \cF_n/\right)\bbZ_k$ the factor $\cF_n/\bbZ_k$ is a nondegenerate focus singularity
(i.e. \ of type $(0,0,1)$). For fixed $k$ and $n$ all such singularities $\cF_n/\bbZ_k$ are Liouville equivalent (see, for example,
\cite{Izosimov11}). Therefore, any two actions of the group $\bbZ_k$ on $\cF_n$ in simple minimal models
can be translated into each other by a  Liouville equivalence of $\cF_n$ (for fixed $k$ and $n$).

It remains to prove the following assertion.

\begin{assertion}\label{A:AtomConjFiniteGroupAction} 
Up to a Liouville automorphism of the $2$-atom $\cV_m$, there are only finitely many actions of the group $\bbZ_k$ on the $2$-atom $\cV_m$.\end{assertion}

The Assertion~\ref{A:AtomConjFiniteGroupAction} follows from the following simple statement.

\begin{assertion}  \label{A:EquivarAct2Atom} 
Let $\rho_1$ and $\rho_2$ be two actions of a finite group $G$ on a $2$-atom $cV_m^2$ that preserve the symplectic structure and the function defining the Liouville foliation (and hence the Liouville foliation itself).
Then, if the actions act in the same way on the set of one-dimensional edges of the singular leaf of the $2$-atom $\cV_m^2$, then they coincide up to
an automorphism of the $2$-atom $\cV_m^2$. In other words, there exists a Liouville equivalence $f:\cV_m^2 \to \cV_m^2$ such that for any $g \in G$
$\rho_2(g) = f^{-1} \circ \rho_1(g)\circ f$.
\end{assertion} 

Essentially the proof of Assertion~\ref{A:EquivarAct2Atom} is similar to the proof of Assertion~8 in \cite{Izosimov11}.

\begin{proof}[\indent Proof of Assertion~\ref{A:EquivarAct2Atom}] 
By cutting the $2$-atom into ``crosses'' and then into smaller pieces, it is easy to construct homeomorphic
fundamental areas for $\rho_1$ and $\rho_2$ actions. The automorphism $f$ identifies these domains and then extends to an equivariant mapping. Assertion~\ref{A:EquivarAct2Atom} is proved.
\end{proof}

\begin{remark} Not all mappings acting in the same way on the set of edges are conjugate. For example, shifts along Hamiltonian vector fields
act identically on the set of edges, but are not conjugate to the identity mapping. We used that the group $G$ is finite, thus there does not exist an automorphism of $\cV_m^2$ of finite order, acting identically on the edges. \end{remark} 

\begin{remark} In Assertion~\ref{A:EquivarAct2Atom} one cannot replace the leafwise homeomorphism $f$ with a symplectomorphism. $2$-atoms have symplectic invariants, and the group $G$ must preserve them.\end{remark}

\begin{remark} The number of edge permutations is finite, so Assertion~\ref{A:AtomConjFiniteGroupAction} can be generalized. Namely, similarly to Assertion~\ref{A:EquivarAct2Atom} it turns out that actions are conjugate if they generate identical subgroups in $\Sym(\cV_m)$. In particular, one can consider $G \subset \Sym(\cV_m)$.
\end{remark}

\subsection{Saddle-focus singularities of complexity $1$, $2$ and $3$}       

Let us now explicitly describe all nondegenerate singularities of saddle-focus type (that is, \ of type $(0,1,1)$) of small complexity.

\begin{theorem} \label{Th:MainComp1} 
Any saddle-focus singularity of complexity $1$ is Liouville equivalent to exactly one of the following $2$ almost direct products:
\begin{equation} \label{Eq:SF_Comp1} 
B \times \cF_1,\qquad (B \times \cF_2)/\bbZ_2.
\end{equation}
\end{theorem}

\begin{remark} 
Both of these singularities were previously described by L.\,M.~Lerman (see~\cite{Lerman2000}). We will give another proof of this fact.
\end{remark}

\begin{proof}[\indent Proof of Theorem~\ref{Th:MainComp1}]  
We apply the algorithm described above.
\begin{enumerate}
\item
Complexity $p=1$. The only divisor of $p$ is $m=1$.
\item
The only focus singularity of complexity $1$ is the $2$-atom $B$. The symmetry group is $\Sym(B)= \bbZ_2$. Therefore $k=1$ or $2$.
\item
For $k=1$ we get the direct product $B \times \cF_1$. The only simple minimal model for $k=2$ is $(B \times \cF_2)/\bbZ_2$.
\end{enumerate}

By Theorem~\ref{Th:SimpleMinModelUnique}, simple minimal models with different groups or factors are not Liouville equivalent. We have received the required list of singularities  \eqref{Eq:SF_Comp1}. Theorem~\ref{Th:MainComp1} is proved.
\end{proof}

\begin{theorem} \label{Th:MainComp2} 
Any saddle-focus singularity of complexity $2$ is Liouville equivalent to exactly one of the following $11$ almost direct products:
\begin{gather*}
B \times \cF_2, \quad (B \times \cF_4)/\bbZ_2, \quad D_1 \times \cF_1, \quad (D_1 \times \cF_2)/\bbZ_2, \quad D_2 \times \cF_1, \\
C_1 \times \cF_1, \quad (C_1 \times \cF_2)/\bbZ_2, \quad (C_1 \times \cF_4)/\bbZ_4, \quad C_2 \times \cF_1,\\
\text{ and two variants of } (C_2 \times \cF_2)/\bbZ_2.
\end{gather*} Two variants of $(C_2 \times \cF_2)/\bbZ_2$ differ in the number of fixed points of rank $0$ for the action of $\bbZ_2$ on $C_2$.
\end{theorem}

\begin{remark} 
In other words, if we represent the $2$-atom $C_2$ as in Fig.~\ref{Fig:C2Sym}, then for one almost direct product
$(C_2 \times \cF_2)/\bbZ_2$ the group $\bbZ_2$ acts as a rotation by $\pi$ around the $z-$axis,
and for the other it acts as a rotation by $\pi$ around the $x-$ or $y-$axis.
\end{remark} 

\begin{proof}[\indent Proof of Theorem~\ref{Th:MainComp2}] 
Theorem~\ref{Th:MainComp2} is proved similarly to Theorem~\ref{Th:MainComp1}.
Let us consider only the case of almost direct products $(C_2 \times \cF_2)/\bbZ_2$. Represent the $2$-atom $C_2$ as in Fig.~\ref{Fig:C2Sym}.
The singularities indicated in the condition of the theorem are Liouville non-equivalent by Corollary~\ref{Cor:DiffFixPoints}. Singularities corresponding to the rotations by $\pi$ around the $x$ and $y$  axes are Liouville equivalent, because the actions are conjugate in the group of automorphisms for the $2$-atom $C_2$. Theorem~\ref{Th:MainComp2} is proved.\end{proof}

Singularities of saddle-focus type of complexity $3$ are classified similarly.
\begin{theorem} \label{Th:MainComp3} 
Any saddle-focus singularity of complexity $3$ is Liouville equivalent to exactly one of the following $21$ almost direct products:\begin{gather*}
B \times \mathcal{F}_3, \quad (B \times \mathcal{F}_6)/\bbZ_2, \quad E_1 \times \mathcal{F}_1, \quad (E_1 \times \mathcal{F}_2)/\bbZ_2, 
\quad (E_1 \times \mathcal{F}_3)/\bbZ_3, \quad  (E_1 \times \mathcal{F}_6)/\bbZ_6, \\
E_2 \times \mathcal{F}_1, \quad (E_2 \times \mathcal{F}_2)/\bbZ_2, \quad E_3 \times \mathcal{F}_1, \quad (E_3 \times \mathcal{F}_2)/\bbZ_2, 
\quad (E_3 \times \mathcal{F}_3)/\bbZ_2, \\
F_1 \times \mathcal{F}_1, \quad F_2 \times \mathcal{F}_1, \quad G_1 \times \mathcal{F}_1, \quad (G_1 \times \mathcal{F}_2)/\bbZ_2, \quad G_2 \times \mathcal{F}_1, \\
 \quad G_3 \times \mathcal{F}_1, \quad (G_3 \times \mathcal{F}_2)/\bbZ_2, \quad H_1 \times \mathcal{F}_1, \quad (H_1 \times \mathcal{F}_3)/\bbZ_3, \quad H_2 \times \mathcal{F}_1.
\end{gather*}  
\end{theorem}

\label{end}


\begin{thebibliography}{99}
\bibitem{Bols91} 
Bolsinov,~A.\,V. 1991, ``Methods of calculation of the Fomenko--Zieschang invariant'',
in Fomenko,~A.\,T. (Ed.), Topological classification of integrable systems (Advances in Soviet Mathematics, Vol.~6), AMS, Providence, pp.~147--183.
\bibitem{BMF} 
Bolsinov,~A.\,V., Matveev,~S.\,V. \& Fomenko,~A.\,T. 1990,
``Topological classification of integrable Hamiltonian systems with two degrees of freedom. List of systems of small complexity'', 
\textit{Russian Math.\ Surv.}, vol.~45, no.~2, pp.~59--94.
\bibitem{BolsFom}   
Bolsinov,~A.\,V. \& Fomenko,~A.\,T. 2004, Integrable Hamiltonian systems: geometry, topology, classification, 
Chapman~\&~Hall\,/\,CRC, Boca Raton, London, N.Y., Washington.
\bibitem{BolsOsh06} 
Bolsinov,~A.\,V. \& Oshemkov,~A.\,A. 2006, ``Singularities of integrable Hamiltonian systems''
in Bolsinov,~A.\,V., Fomenko,~A.\,T. \& Oshemkov,~A.\,A. (Eds.), Topological Methods in the Theory of Integrable Systems, 
Cambridge Scientific Publishers, Cambridge, pp.~1--67.
\bibitem{Eliasson90} 
Eliasson,~L.\,H. 1990, ``Normal forms for Hamiltonian systems with Poisson commuting integrals --- elliptic case'',
\textit{Comment.\ Math.\ Helv.}, vol.~65, no.~1, pp.~4--35.
\bibitem{Fom86} 
Fomenko,~A.\,T. 1987, ``The topology of surfaces of constant energy in integrable Hamiltonian systems, and obstructions to integrability'', 
\textit{Math.\ USSR-Izv.}, vol.~29, no.~3, pp.~629--658.
\bibitem{Fom88} 
Fomenko,~A.\,T. 1988, ``Topological invariants of Liouville integrable Hamiltonian systems'',
\textit{Funct.\ Anal.\ Appl.}, vol.~22, no.~4, pp.~286--296.
\bibitem{Izosimov11} 
Izosimov,~A.\,M. 2011, ``Classification of almost toric singularities of Lagrangian foliations'', 
\textit{Sb.\ Math.}, vol.~202, no.~7, pp.~1021--1042.
\bibitem{LermanUmankii87}  
Lerman,~L.\,M. \& Umanskii,~Ya.\,L. 1987; 1988, ``Structure of the Poisson action of $\bbR^2$ on a four-dimensional symplectic manifold. I; II'',
\textit{Selecta Math.\ Sov.}, vol.~6, pp.~365--396; vol.~7, pp.~39--48.
\bibitem{LermanUmankii92}
Lerman,~L.\,M. \& Umanskii,~Ya.\,L. 1994; 1994; 1995, ``Classification of four-dimensional integrable Hamiltonian systems and Poisson actions 
of $R^2$ in extended neighborhoods of simple singular points. I; II; III'',
\textit{Sb.\ Math.}, vol.~77, no.~2, pp.~511--542; vol.~78, no.~2, pp.~479--506; vol.~186, no.~10, pp.~1477--1491.
\bibitem{Lerman2000}  
Lerman,~L.\,M. 2000, ``Isoenergetical Structure of Integrable Hamiltonian Systems in an Extended Neighborhood of a Simple Singular Point: 
Three Degrees of Freedom'', \textit{Amer.\ Math.\ Soc.\ Transl. (2)}, vol.~200, pp.~219--242.
\bibitem{Matveev96}
Matveev,~V.\,S. 1996, ``Integrable Hamiltonian system with two degrees of freedom. The topological structure of saturated neighbourhoods 
of points of focus-focus and saddle-saddle type'', \textit{Sb.\ Math.}, vol.~187, no.~4, pp.~495--524.
\bibitem{MatOsh}
Matveev,~V.\,S. \& Oshemkov,~A.\,A. 1999, ``Algorithmic classification of invariant neighborhoods for points of saddle-saddle type'', 
\textit{Moscow Univ.\ Math.\ Bull.}, vol.~54, no.~2, pp.~44--47. 
\bibitem{OshMiran}
Oshemkov, A.\,A. 1995, ``Morse functions on two-dimensional surfaces. Encoding of singularities'',
\textit{Proc.\ Steklov Inst.\ Math.}, vol.~205, pp.~119--127.
\bibitem{Oshemkov10}
Oshemkov, A.\,A. 2010, ``Classification of hyperbolic singularities of rank zero of integrable Hamiltonian systems'', 
\textit{Sb.\ Math.}, vol.~201, no.~8, pp.~1153--1191.
\bibitem{Oshemkov11}
Oshemkov,~A.\,A. 2011 ``Saddle singularities of complexity 1 of integrable Hamiltonian systems'',
\textit{Moscow Univ.\ Math.\ Bull.}, vol.~66, no.~2, pp.~60--69. 
\bibitem{Zung1} 
Nguyen Tien Zung. 1996, ``Symplectic topology of integrable Hamiltonian systems. I: Arnold--Liouville with singularities'',
\textit{Compositio Math.}, vol.~101, pp.~179--215.
\bibitem{ZungFFNote} 
Nguyen Tien Zung. 1997, ``A note on focus-focus singularities'',
\textit{Diff.\ Geom.\ and Appl.}, vol.~7, pp.~123--130.

\bibitem{KudrOsh20} E. A. Kudryavtseva, A. A. Oshemkov, “Bifurcations of integrable mechanical systems with magnetic field on surfaces of revolution”, \textit{Chebyshevskii Sb.}, \textbf{21}:2 (2020), 244--265.
\end{thebibliography}
\end{document}